\documentclass[12pt,reqno]{amsart}
\usepackage{amsfonts}
\usepackage{amssymb}
\usepackage{amsmath}
\usepackage{latexsym}
\usepackage{amsthm}
\usepackage{epsfig}
\usepackage{color}
\usepackage{amsfonts,amssymb,mathrsfs,amscd}
\usepackage{comment}

\newtheorem{theorem}{Theorem}[section]

\newtheorem{lemma}[theorem]{Lemma}

\theoremstyle{definition}
\newtheorem{definition}[theorem]{Definition}
\newtheorem{remark}[theorem]{Remark}
\numberwithin{equation}{section}

\newcommand{\curl}{\mathop\mathrm{curl}}
\renewcommand{\div}{\mathop\mathrm{div}}
\newcommand{\dist}{\operatorname{dist}}

\def\Inf{\operatornamewithlimits{inf\vphantom{p}}}

\begin{document}

\title[ Damped  Navier--Stokes system] {Vanishing viscosity limit for global
attractors for the damped  Navier--Stokes system with
stress free boundary conditions}
\author[V.V. Chepyzhov, A.A. Ilyin, and S.V Zelik]
{Vladimir Chepyzhov${}^{1,3}$, Alexei Ilyin ${}^{1,2}$ and Sergey
Zelik${}^{2,4}$}

\dedicatory{To Edriss Titi on the occasion of his 60-th birthday
with warmest regards }

\begin{abstract}
We consider  the damped and driven Navier--Stokes system with stress free
boundary conditions and the damped
Euler system in a bounded domain  $\Omega\subset\mathbf{R}^2$.
We show that the damped Euler system has a (strong)
global attractor in~$H^1(\Omega)$. We also
show that
 in the vanishing viscosity limit the global attractors of the
Navier--Stokes system converge in the non-symmetric Hausdorff distance
in $H^1(\Omega)$ to the the strong global attractor
of the limiting damped Euler system (whose solutions are not necessarily unique).
\end{abstract}


\thanks{The research of V.~Chepyzhov and A.~Ilyin was carried out in the
Institute for Information Transmission Problems, Russian Academy of
Sciences, at the expense of the Russian Science Foundation (project
14-50-00150). The work of S.~Zelik was supported in part by
the RFBR grant 15-01-03587}

\address{
 \newline ${}^1$ Institute for Information Transmission Problems, Moscow 127994, Russia,
 \newline ${}^2$ Keldysh Institute of Applied Mathematics, Moscow 125047, Russia,
 \newline ${}^3$ National Research University Higher School of Economics, Moscow 101000,
 Russia,
 \newline ${}^4$ University of Surrey, Department of Mathematics, Guildford, GU2 7XH, UK.}

\email{chep@\,iitp.ru}

\email{ilyin@\,keldysh.ru}

\email{s.zelik@\,surrey.ac.uk}

\maketitle

\section{Introduction
\label{Sec0}}

In this paper, we study from the point of view of global attractors  the 2D damped and driven
Navier--Stokes system
\begin{equation} \label{NS}
\left\{
\begin{array}{l}
\partial _{t}u+(u,\nabla )u+\nabla p+ru=\nu \Delta u+g(x), \\
\div u=0,\quad u(0)=u^{0},
\end{array}
\right.
\end{equation}
and the corresponding limiting ($\nu =0$) damped/driven Euler
system
\begin{equation}\label{E}
\left\{
\begin{array}{l}
\partial _{t}u+(u,\nabla )u+\nabla p+ru=g(x), \\
\div u=0,\quad u(0)=u^{0}.
\end{array}
\right.
\end{equation}
Both systems are considered in a bounded multiply connected smooth
domain $\Omega \subset \mathbb{R}^{2}$ with
standard non-penetration boundary condition
\begin{equation}\label{bcE}
u\cdot n|_{\partial \Omega }=0,
\end{equation}
while the system \eqref{NS} is supplemented with the so-called
stress free or slip boundary conditions
\begin{equation}\label{bcNS}
u\cdot n\vert_{\partial\Omega}=0,\qquad\curl u\vert_{\partial\Omega}=0.
\end{equation}
The Laplace operator with \eqref{bcNS} commutes
the Leray projection. These boundary conditions
 guarantee the absence of the boundary layer and yield
the conservation of enstrophy
in the unforced and undamped case of~\eqref{E}. They are also
convenient for studying the limit as $\nu\to0^+$ of the
individual solutions of the 2D Navier--Stokes system~\cite{Bar, Lio}.

Systems \eqref{NS} and \eqref{E} are relevant  in geophysical
hydrodynamics and the  damping term $-ru$ describes the Rayleigh or
Ekman friction and parameterizes the main dissipation occurring in
the planetary boundary layer (see, for example, \cite{Ped}).
The viscous  term $-\nu \Delta u$ in system \eqref{NS} is
responsible for the small scale dissipation. We also observe that
in physically relevant cases we have $\nu \ll r|\Omega|$.

The damped and driven 2D Euler and Navier-Stokes systems attracted
considerable attention over the last years and were studied from
different points of view. The regularity, uniqueness, and stability
of the stationary solutions for \eqref{E}
 were studied in \cite{BCT, Sau, W}. The vinishing viscosity limit for system
\eqref{NS} was studied for steady-state statistical solutions in~\cite{CoRa}.

In the presence of the damping term the  weak attractor for the
system \eqref{E} was
 constructed in \cite{IlyinEU} in the phase space $H^1$.  In the trajectory phase
space the weak attractor was constructed in~\cite{BF, CVrj}.

The dynamical effects of the damping term $-ru$ in the case of the
Navier--Stokes system~\eqref{NS} were studied in~\cite{IMT} on the
torus, on the 2D sphere, and in bounded (simply connected) domain
$\Omega$ with boundary conditions  \eqref{bcNS}.
Specifically, it was shown that the fractal dimension of the global
attractor $\mathscr A_\nu$ satisfies the estimate
\begin{equation}\label{min}
\dim_f\mathscr A_\nu\le \min
\left(c_1(\Omega)\,
\frac{\|\curl g\|\,|\Omega|^{1/2}}{\nu r}\,,\
c_2(\Omega)\,\frac{\|\curl g\|^2}{\nu r^3}\right),
\end{equation}
where $|\Omega|$ is the area of the spatial domain. This estimate is sharp
in the limit $\nu\to 0^+$ and the lower bound
is provided by the corresponding  family of Kolmogorov flows.
Furthermore, the constants $c_1$ and $c_2$ are given explicitly
for the torus $\Omega=\mathbb{T}^2$ and for the sphere $\Omega=\mathbb{S}^2$.
The case of an elongated torus $\mathbb{T}^2_\alpha$ with periods
$L$ and $L/\alpha$, where $\alpha\to 0^+$ was studied in
\cite{I-T-JMFM}, where it was shown that~\eqref{min} still holds
for $\mathbb{T}^2_\alpha$ and is sharp as both $\alpha\to0$ and $\nu\to0$.

The essential analytical tool used in the proof of~\eqref{min},
especially in finding explicit values of $c_1$ and $c_2$,
is the Lieb--Thirring inequality. New bounds for the
Lieb--Thirring constants for the anisotropic torus were
recently obtained in \cite{IL} with applications to the
system~\eqref{NS} on~$\mathbb{T}^2_\alpha$.

One might expect that in the case of the damped Navier--Stokes system~\eqref{NS} in
$\mathbb{R}^2$ in the space of finite energy solutions the attractor
$\mathscr A_\nu$ exists and its fractal dimension is
bounded by the second  number on the right-hand side in \eqref{min}.
It was recently shown in \cite{IPZ} that it is indeed the case:
\begin{equation}\label{R2}
\dim_f\mathscr A_\nu\le \frac{1}{16\sqrt{3}}
\frac{\|\curl g\|^2}{\nu\,r^{3}}.
\end{equation}
Moreover, due to convenient scaling available for $\mathbb{R}^2$ this estimate
 of the dimension is included  in \cite{IPZ}
in the family of estimates depending on the norm of $g$ in the scale
of homogeneous Sobolev spaces $\dot H^s(\mathbb{R}^2)$, $-1\le s\le 1$;
the case $s=1$ being precisely~\eqref{R2}.

Estimates for the degrees of freedom for the damped Navier--Stokes system~\eqref{NS}
expressed in terms of various finite dimensional projections were
obtained in~\cite{IT}. They are also of the order~\eqref{min}.

We point out two important differences between the damped
Navier--Stokes system \eqref{NS} and the damped Euler system~\eqref{E}
which make the construction of the global attractor for~\eqref{E}
less straightforward. The first  is the absence for~\eqref{E} of the
instantaneous smoothing property of solutions and explains why the existence
of only a weak attractor was first established \cite{IlyinEU}.  The second  is that the uniqueness is only known for the
solutions with bounded vorticity~\cite{Yud2}
and is  not known in
the natural Sobolev space ${H}^{1}$, which makes the
trajectory attractors very convenient for \eqref{E},
see \cite{BF, CVrj, CVbook, CV2, VC-UMN}.
The
trajectory attractors for~\eqref{E} in the weak topology of ${H}^{1}$
were constructed in \cite{CVrj} (see also \cite{BF}) and in
\cite{Ch15} for the non-autonomous case. In addition,
the upper semi-continuous dependence as $\nu\to 0+$ of the trajectory attractors of
the system \eqref{NS} on the torus was established in \cite{CVrj}  in the weak
topology of $H^{1}(\Bbb{T}^{2})$.

The existence of the strong ${H}^{1}$ trajectory attractors for the dissipative Euler system
\eqref{E} on the 2D torus was proved in~\cite{CVZ} under the assumption that
$\curl g\in L^{\infty }$ which was used to prove the enstrophy equality. The strong attraction and compactness for the
trajectory attractor were established using the energy method developed in
\cite{B, Gh, MRW, Rosa} for the equations in unbounded, non-smooth domains or
for equations without uniqueness. This method is based on the corresponding energy
 balance  for the solutions
and leads
to the asymptotic compactness of the solution semigroups or collections the
trajectories.

Most closely related to the present work is the paper \cite{CIZ}
where the strong global and trajectory $W^{1,p}$-attractors were constructed for the
system \eqref{E} in $\mathbb{R}^{2}$. The crucial equation of the enstrophy balance
is proved there in the Sobolev spaces $W^{1,p}$, $2\le p<\infty$ without the assumption on $g$ that
guarantees the uniqueness of solutions on the attractor. Instead the authors
used the fact that in the 2D case the vorticity satisfies a scalar transport
equation, and the required enstrophy equality directly follows from the
results of \cite{DiP-Lio}.

In unbounded domains the damped Navier--Stokes and Euler systems can be studied from the point of view  of uniformly local spaces (where the energy is infinite) and  one of the main issues is the
proof of the dissipative estimate, which is achieved by means of delicate
 weighted estimates. In  the uniformly local spaces in the
viscous case $\nu>0$ the global attractors for \eqref{NS}  in the strong
topology were constructed in~\cite{Ze13}, see also~\cite{Ze07,Ze8} for similar
results in channel-like domains. In the inviscid case the strong attractor for
\eqref{E} in the uniformly local $H^1$ space was recently constructed in
\cite{CZ15}.

In the present  paper we study the convergence of the global attractors
$\mathscr{A}_{\nu }$ of the system \eqref{NS},
\eqref{bcNS} in the vanishing  viscosity limit $\nu \rightarrow 0^+$,
and our main result is as follows.
 The system \eqref{E}, \eqref{bcE} has a global
attractor $\mathscr{A}_{0}\Subset H^1(\Omega)$. For every
$\delta$-neighbourhood $\mathscr{O}_\delta$
of $\mathscr{A}_{0}$ in $H^1(\Omega)$
there exits $\nu(\delta)>0$ such that
\begin{equation}\label{convergence}
\mathscr{A}_{\nu}\subset \mathscr{O}_\delta(\mathscr{A}_{0})\qquad\text{for all}
\quad\nu\le\nu(\delta),
\end{equation}
where $\mathscr{A}_{\nu}$ for $\nu>0$ are the attractors of the damped Navier--Stokes system~\eqref{NS}, \eqref{bcNS}.

We point out that
despite the fact that the dimension of $\mathscr A_\nu$ can be
of order $1/\nu$ as $\nu\to 0^+$ (at least in the periodic case and the special family
Kolmogorov-type forcing terms) the limiting attractor $\mathscr A_0$
is, nonetheless, a compact set in $H^1(\Omega)$.

This paper has the following structure. In Section~\ref{Sec1} we define the function
spaces, paying attention the  case when the domain
$\Omega$ is multiply connected, and
construct the global
attractors $\mathscr{A}_{\nu }$ for~\eqref{NS}, \eqref{bcNS}. In Section~\ref{Sec2}
we prove the existence of weak solutions
of the damped  Euler system~\eqref{E}.
We adapt the theory of renormalized solutions from~\cite{DiP-Lio}
to the vorticity equation in a  bounded domain
which gives us the crucial equation of
the enstrophy balance
for an arbitrary weak solution of~\eqref{E}. In Section~\ref{Sec3} we consider the
generalized solution semigroup for the system \eqref{E} and define weak and strong
global attractors for the generalized semigroup. We first construct a
weak global attractor $\mathscr{A}_{0}$ in ${H}^{1}$ for \eqref{E}
and then we prove the asymptotic compactness of the generalized semigroup
which gives that the weak global attractor $\mathscr{A}_{0}$ is,
in fact, the ${H}^{1}$ strong global attractor. In Section~\ref{Sec4} we prove
\eqref{convergence}.

\section{Equations and function spaces}
\label{Sec1}

We shall be dealing with  the   damped and driven  Navier--Stokes system
\eqref{NS} with boundary conditions \eqref{bcNS}
and the corresponding limiting ($\nu=0$) damped Euler system~\eqref{E}
with standard non-penetration condition~\eqref{bcE}.

Both systems are studied in  a bounded domain $\Omega\subset\mathbb{R}^2$.
We consider the general case when $\Omega$ can be multiply connected
with boundary
$$
\partial\Omega=\Gamma=\Gamma_0\cup\Gamma_1\cup\dots\cup\Gamma_k.
$$
In other words, $\Gamma_0$ is the outer boundary, and the $\Gamma_i$'s are
the boundaries of $k$ islands inside $\Gamma_0$. We assume that $\partial \Omega$ is
smooth ($C^2$ will be enough) so that there exists a well-defined
outward unit normal $n$ and also an extension operator $E$:
$$
E:H^2(\Omega)\to H^2(\mathbb{R}^2),\qquad \|Eu\|_{H^2(\mathbb{R}^2)}\le
\mathrm{const}\|u\|_{H^2(\Omega)}.
$$

We now introduce  the required function spaces and  their
orthogonal decompositions. We set
$$
\mathcal{H}=\{u\in \mathbf{L}^2(\Omega),\ \div u=0,\ u\cdot n\vert_{\partial\Omega}=0\}.
$$
The following orthogonal decomposition holds \cite[Appendix 1]{TemNS}:
\begin{equation}\label{H}
\mathcal{H}=H_0\oplus H_c,
\end{equation}
where
$$
H_0=\{u\in \mathbf{L}^2(\Omega),\ \div u=0,
\ u\cdot n\vert_{\partial\Omega}=0, \ u=\nabla^\perp\varphi, \ \varphi\in H^1_0(\Omega)\},
$$
that is, the vector functions in $H_0$ have a unique single valued
stream function $\varphi$ vanishing at all components of the
boundary $\partial\Omega$. Here $\varphi$ is a scalar function, and
$$
\nabla^\perp\varphi:=\{-\partial_2\varphi,\partial_1\varphi\}=-\curl\varphi,
\quad  u^\perp:=\{-u^2,\,u^1\}.
$$
Accordingly, the orthogonal complement to $H_0$ in $\mathcal{H}$ is
the $k$-dimensional space of harmonic (and hence infinitely smooth) vector functions:
$$
H_c=\{u\in \mathbf{L}^2(\Omega),\ \div u=0,\ \curl u=0, \ u\cdot n\vert_{\partial\Omega}=0 \},
$$

In the similar way, for smoothness of order one  we have
$$
\mathcal{H}^1:=\{u\in \mathbf{H}^1(\Omega),\ \div u=0,\ u\cdot n\vert_{\partial\Omega}=0\}=
H_1\oplus H_c,
$$
where $H_c$ is as before and
$$
H_1=\{
 u=\nabla^\perp\varphi, \ \varphi\in H^2(\Omega)\cap H^1_0(\Omega)\},\qquad
\|u\|_{H_1}=\|\Delta \varphi\|.
$$
For smoothness of order two
$$
\mathcal{H}^2:=\{u\in \mathbf{H}^2(\Omega),\ \div u=0,\ u\cdot n\vert_{\partial\Omega}=0\}=
H_2\oplus H_c,
$$
where
$$
H_2=\{
 u=\nabla^\perp\varphi, \ \varphi\in H^3(\Omega)\cap H^1_0(\Omega)\},\qquad
\|u\|_{H_2}=\|\nabla\Delta \varphi\|.
$$

Corresponding to the second boundary condition in~\eqref{bcNS}
is the following closed subspace in $H_2$:
$$
H_2^0=\{
 u=\nabla^\perp\varphi, \ \varphi\in H^3(\Omega)\cap H^1_0(\Omega)
 \cap\{\Delta\varphi\vert_{\partial\Omega}=0\}\}.
$$
The space of all divergence free vector functions of class $\mathbf{H}^2(\Omega)$
satisfying the boundary conditions~\eqref{bcNS} is denoted by $\mathcal{H}^2_0$:
\begin{equation}\label{H20}
\mathcal{H}^2_0=H^0_2\oplus H_c.
\end{equation}

The orthonormal basis in $H_0$ is made up of vector functions
$$
u_j=\lambda_j^{-1/2}\nabla^\perp\varphi_j,
$$
where $\lambda_j$ and $\varphi_j$ are the eigenvalues and eigenfunctions of
the scalar Dirichlet Laplacian \cite{I93}
$$
-\Delta\varphi_j=\lambda_j\varphi_j,\quad\varphi_j\vert_{\partial\Omega}=0,\quad
0<\lambda_1<\lambda_2\le\lambda_3\dots\to+\infty.
$$
In fact,
$$
\|u_j\|^2=\lambda_j^{-1}(\nabla^\perp\varphi_j,\nabla^\perp\varphi_j)=
\lambda_j^{-1}\|\nabla\varphi_j\|^2=1.
$$
Furthermore, since on scalars
\begin{equation}\label{curlonscalars}
\curl\nabla^\perp=-\curl\curl=\Delta,\qquad \curl=-\nabla^\perp,
\end{equation}
the $u_j$'s satisfy~\eqref{bcNS}, and
the system $\{u_j\}_{j=1}^\infty$ is the complete
orthonormal basis of eigen vector functions with eigenvalues
$\{\lambda_j\}_{j=1}^\infty$ of the vector Laplacian
\begin{equation}\label{VecLap}
\Delta=\nabla\div-\curl\curl
\end{equation}
with boundary conditions~\eqref{bcNS}:
$$
-\Delta u_j=\curl\curl u_j=\lambda_j u_j.
$$

We can express the fact that a vector function $u$ belongs to
$H_0$, $H_1$, or $H_2^0$ in terms of its Fourier coefficients as follows.
Let
\begin{equation}\label{sp_basis}
u=\sum_{j=1}^\infty c_ju_j,\qquad c_j=(u,u_j)=\lambda_j^{-1/2}(u,\nabla^\perp\varphi_j),
\end{equation}
where (setting $\omega:=\curl u$)
$$
(u,\nabla^\perp\varphi_j)=
(u^\perp,\nabla\varphi_j)=-(\div u^\perp,\varphi_j)=(\curl u,\varphi_j)=
(\omega,\varphi_j).
$$
This gives that
$$
\aligned
u\in H_0&\Leftrightarrow \sum_{j=1}^\infty c_j^2=\|u\|^2=\|\omega\|^2_{H^{-1}(\Omega)}<\infty,\\
u\in H_1&\Leftrightarrow \sum_{j=1}^\infty \lambda_j c_j^2=\|\omega\|^2<\infty,\\
u\in H_2^0&\Leftrightarrow \sum_{j=1}^\infty \lambda_j^2c_j^2=\|\curl\curl u\|^2=\|\omega\|^2_{H^{1}_0(\Omega)}<\infty.
\endaligned
$$

The basis in the $k$-dimensional space of harmonic vector functions $H_c$
is given in~\cite[Appendix 1, Lemma~1.2]{TemNS} in terms of the gradients
of harmonic multi valued functions. In our 2D case it is more convenient
to construct a basis in $H_c$ in terms of single valued stream functions.

\begin{lemma} \label{L:basis}
The system $\{\nabla^\perp\psi^j\}_{j=1}^k$ is a basis in $H_c$.
Here $\psi^j$ is the solution in  $\Omega$ of the equation
$\Delta\psi^j=0$, where $\psi^j=0$ at all the components of the
boundary $\Gamma$ except for $\Gamma_j$, where $\psi^j=1$.
\end{lemma}
\begin{proof} The vector functions  $\nabla^\perp\psi^j\in H_c$ and are linearly independent.
\end{proof}

Next, we consider the Leray projection $P$ from $\mathbf{L}^2(\Omega)$ onto
$\mathcal{H}$. In accordance with~\eqref{H} we have
$P=P_0\oplus P_c$. For the projection $P_0$ onto $H_0$ we have
\begin{equation}\label{P0}
P_0u=\nabla^\perp(\Delta^D_\Omega)^{-1}\curl u,
\end{equation}
where $\Delta^D_\Omega$ is the (scalar) Dirichlet Laplacian, which
is an isomorphism from $H^1_0(\Omega)$ onto $H^{-1}(\Omega)$.

\begin{lemma} \label{L:commut}  
On $\mathcal{H}^2_0$ the projection $P$ commutes with the Laplacian
$\Delta$ with
boundary conditions~\eqref{bcNS}.
\end{lemma}
\begin{proof}
Since $P_c\Delta=\Delta P_c=0$  on $H_c$, it suffices to consider
$P_0$. Let $u\in {H}^0_2$, see~\eqref{H20}, so that $P_0u=u$. Then
interpreting $\curl u$ as a scalar and using~\eqref{curlonscalars} we obtain
$$
\aligned
P_0\Delta u=-\nabla^\perp(&\Delta^D_\Omega)^{-1}\curl\curl\curl u=
\nabla^\perp(\Delta^D_\Omega)^{-1}\Delta\curl u=\\=
&\nabla^\perp\curl u=-\curl\curl u=\Delta u=\Delta P_0u.
\endaligned
$$
\end{proof}

This lemma makes the subsequent analysis very similar to the 2D periodic case
or the case of a manifold without boundary.

We also recall
the familiar formulas
\begin{equation}\label{fam}
\aligned
&(\nabla\varphi,v)=-(\varphi, \div v),\quad &v\cdot n\vert_{\partial\Omega}=0,\\
&(\curl \varphi,v)=(\varphi, \curl v),\quad &\curl v\vert_{\partial\Omega}=0.
\endaligned
\end{equation}

\begin{lemma} \label{L:orth} \cite{I93}
Let $u\in\mathcal{H}^0_2$ $(see~\eqref{H20})$. Then
\begin{equation}\label{orth}
((u,\nabla)u,\Delta u)=0.
\end{equation}
\end{lemma}
\begin{proof}
We use the invariant expression for the convective term
$$
(u,\nabla)u=\curl u\times u+\frac12\nabla u^2.
$$
Let $u=u_0+u_c$, where $u_0\in H^0_2$, $u_c\in H_c$. Then,
taking into account~\eqref{VecLap}, for
the second term in the above expression we have
$$
(\nabla u^2,\curl\curl u_0)=(\curl\nabla u^2,\curl u_0)=0,
$$
since $\curl\nabla=0$ algebraically, and the first equality follows
from \eqref{fam} with boundary condition $\curl u_0\vert_{\partial\Omega}=0$

For the first term we have setting $\omega=\curl u_0$
and using~\eqref{curlonscalars}
$$
\aligned
(\curl u_0\times u, \curl\curl u_0)=
-(\omega u^\perp, \nabla^\perp\omega)=\\
-(\omega u,\nabla\omega)=
-\frac12(u,\nabla \omega^2)=\frac12(\div u,\omega^2)=0,
\endaligned
$$
where we used $u\cdot n\vert_{\partial\Omega}=0$ for the
integration by parts.
\end{proof}

We also recall the familiar orthogonality relation
\begin{equation}\label{b}
b(u,v,v)=0,
\end{equation}
where the trilinear form $b$
$$
b(u,v,w)=\int_\Omega\sum_{i,j=1}^2u^i\partial_i v^jw^jdx
$$
is continuous on $\mathcal{H}^1$.

The space $H_c$ of (infinitely smooth) harmonic vector functions
is $k$-dimensional, and every Sobolev norm $\mathbf{H}^k(\Omega)$
is equivalent to the $\mathbf{L}^2(\Omega)$-norm.
Therefore the $\mathbf{H}^1(\Omega)$-norm on $\mathcal{H}^1$
for $u=u_0+u_c\in H_1\oplus H_c$
can be given by
$$
\|u\|_1^2:=\|u\|^2+\|\curl u\|^2=\|u\|^2+\|\curl u_0\|^2.
$$
Accordingly,  the $\mathbf{H}^2(\Omega)$-norm on $\mathcal{H}^2_0$
is given by
$$
\|u\|_2^2:=\|u\|^2+\|\curl\curl u\|^2=\|u\|^2+\|\curl\curl u_0\|^2.
$$

\begin{theorem} \label{T:NSEE}
Let the initial data $u^0$  and the right-hand side $g$
in the damped Navier--Stokes system~\eqref{NS}, \eqref{bcNS}
satisfy
$$
u^0\in\mathcal{H}^1,\qquad g\in \mathcal{H}^1.
$$
Then there exists a unique strong solution
$u\in C([0,T];\mathcal{H}^1)\cap L^2(0,T;\mathcal{H}^2_0)$
of~\eqref{NS}, \eqref{bcNS}.
Thus, a semigroup of solution operators
$$
u(t)=S(t)u(0),
$$
corresponding to
\eqref{NS}, \eqref{bcNS} is well defined.

The solution satisfies
the equation of balance of energy and enstrophy:
\begin{equation}\label{NSee}
\frac12\frac{d}{dt}\|u\|_1^2+\nu\|\Delta u\|^2+r\|u\|_1^2=(g,u)_1,
\end{equation}
where
$$
(g,u)_1:=(g,u)+(\curl g,\curl u).
$$
\end{theorem}
\begin{proof} The proof is standard and uses the Galerkin method.
We use the special basis~\eqref{sp_basis} in $\mathcal{H}^2_0\subset\mathcal{H}^1$
and supplement it with a $k$-dimensional basis in $H_c$,
for example, with the one from Lemma~\ref{L:basis} starting the enumeration
from the basis in $H_c$.

Then for every approximate Galerkin solution
$$
u=u^{(n)}=\sum_{k=1}^nc_ku_k\in \mathcal{H}^2_0
$$
we have the orthogonality relations~\eqref{orth}, \eqref{b}.
We take the scalar product of~\eqref{NS} with $u$,
 and also
 with $\Delta u$, integrate by parts using~\eqref{fam},
 drop  the $\nu$-terms
and  use  Growwall's inequality to
  obtain in the standard way the   estimates
  $$
  \aligned
  &\|u(t)\|^2\le\|u(0)\|e^{-rt}+r^{-2}\|g\|^2,\\
  &\|\curl u(t)\|^2\le\|\curl u(0)\|e^{-rt}+r^{-2}\|\curl g\|^2,
  \endaligned
  $$
  which gives
\begin{equation}\label{apr}
\|u(t)\|_1^2\le\|u(0)\|_1e^{-rt}+r^{-2}\|g\|_1^2
\end{equation}
for $u=u^{(n)}$, uniformly for $n$ and $\nu>0$. The remaining assertions
of the theorem are proved very similarly to the classical case
of the 2D Navier--Stokes system with Dirichlet boundary
conditions (even simpler, since we now  have
more regularity, see, for instance, \cite{BV},\cite{TemNS}).
\end{proof}

We recall the following definition of the (strong) global
attractor (see, for instance \cite{BV},\cite{Tem}).
\begin{definition}\label{D:attr} Let $S(t)$, $t\ge0$, be a semigroup acting in a Banach space $\mathcal B$.
Then the set $\mathscr A\subset\mathcal B$ is a global attractor of  $S(t)$ if
\par
1)  $\mathscr A$ is compact in $\mathcal B$: $\mathscr A\Subset\mathcal B$.
\par
2) $\mathscr A$ is strictly invariant: $S(t)\mathscr A=\mathscr A$.
\par
3) $\mathscr A$ is globally attracting, that is,
$$
\lim_{t\to\infty}\mathrm{dist}(S(t)B,\mathscr A)=0,
\quad\text{for every bounded set $B\subset\mathcal{B}$}.
$$
\end{definition}

\begin{theorem}\label{T:AttrNS} The semigroup $S(t)$ corresponding to
\eqref{NS}, \eqref{bcNS} has a global attractor $\mathscr{A}\Subset\mathcal{H}^1$.
\end{theorem}
\begin{proof}
It follows from~\eqref{apr} that the ball
\begin{equation}\label{ball}
B_0=\{u\in\mathcal{H}^1,\ \ \|u\|_1^2\le2r^{-2}\|g\|_1^2\}
\end{equation}
is the absorbing ball for $S(t)$.  The semigroup $S(t)$ is continuous in
$\mathcal{H}^1$ and  has the smoothing property
(which can be shown similarly to  the classical 2D Navier--Stokes system
\cite{BV}, \cite{Tem}). Therefore the set
$$
B_1=S(1)B_0
$$
is a compact absorbing set, which gives the existence
of the attractor $\mathscr{A}\Subset\mathcal{H}^1$. We finally point out
that for $u(t)\in\mathscr A$ we have for all $t\in \mathbb{R}$
\begin{equation}\label{estonA}
\|u(t)\|_1\le r^{-1}\|g\|_1
\end{equation}
uniformly with respect to $\nu>0$.
\end{proof}

\section{Weak solutions for the  Euler system and energy-enstrophy balance}
\label{Sec2}

We now turn to the damped and driven Euler system~\eqref{E}, \eqref{bcE}.
\begin{definition}\label{D:weakE} Let $u(0),g\in\mathcal{H}^1$. A vector function
 $u=u(t,x)$ is called a weak solution of~\eqref{E}, \eqref{bcE} if
 $u\in L^\infty(0,T;\mathcal{H}^1)$ and satisfies the integral identity
 \begin{equation}\label{identE}
\aligned
-\int_0^T(u,v\eta'(t))dt+&\int_0^Tb(u,u,v\eta(t))dt+\\
+r&\int_0^T(u,v\eta(t))dt=\int_0^T(g,v\eta(t))dt
\endaligned
 \end{equation}
 for all $\eta\in C^\infty_0(0,T)$ and all $v\in\mathcal{H}^1$.
\end{definition}

\begin{theorem}\label{T:WeakE}
There exists at least one solution of the damped Euler
system~\eqref{E}, \eqref{bcE}. Moreover, every
weak solution in the sense of Definition~\ref{D:weakE}
is of class $C([0,T]; \mathcal{H})$ and satisfies
the equation of balance of energy
\begin{equation}\label{balEnergy}
\frac12\frac d{dt}\|u(t)\|^2+r\|u(t)\|^2=(g,u(t)).
\end{equation}
\end{theorem}
\begin{proof} As before we use the special basis
and see that approximate Galerkin solutions $u^{n}$ satisfy~\eqref{apr}
and therefore we obtain that uniformly with respect to $n$
$$
u^{n}\in L^\infty(0,T;\mathcal{H}^1).
$$
Next, we see from  equation~\eqref{E} that
$\partial_tu^{n}$ is bounded in $L^2(0,T;\mathcal{H}^{-1})$.
Therefore we can extract a subsequence (still denoted by $u^{n}$)
such that
$$
u^{n}\to u\ \text{$*$-weakly in}\ L^\infty(0,T;\mathcal{H}^1)\ \text{and strongly in}\
L^2(0,T;\mathcal{H}).
$$
This is enough to pass to the limit in the non-linear term in
\eqref{identE} and therefore to verify that $u$
satisfies~\eqref{identE}. Since $\partial_tu\in L^2(0,T;\mathcal{H}^{-1})$,
it follows that we can take the scalar product of~\eqref{E}
with the solution $u$ to obtain~\eqref{balEnergy}, see~\cite{TemNS}.
\end{proof}

We now derive the scalar  equation for $\omega=\curl u$.
We set in~\eqref{identE}
$$
v=\curl\varphi,\qquad \varphi\in C_0^\infty(\Omega)
$$
and integrate by parts the linear terms in~\eqref{identE}
by using the second formula in~\eqref{fam}. For the non-linear term we have
\begin{equation}\label{eqvort}
\aligned
b(u,u,v)=\int_\Omega(u,\nabla)u\cdot\curl\varphi dx=
\int_\Omega(\curl u\times u)\cdot\curl\varphi dx=\\
\int_\Omega(\omega u^\perp)\cdot\curl\varphi dx=
\int_\Omega\curl(\omega u^\perp)\varphi dx=
\int_\Omega u\nabla \omega \varphi dx,
\endaligned
\end{equation}
since algebraically $\curl(\omega u^\perp)=\omega\div u+u\nabla \omega$.

Thus, we have shown that $\omega =\curl u$ satisfies in $\Omega$
the following equation (in the sense of distributions)
\begin{equation}\label{omegainOmega}
\aligned
\partial_t\omega+u\nabla\omega+r\omega=G,\\
\omega(0)=\omega^0.
\endaligned
\end{equation}
where $G=\curl g$, $\omega^0=\curl u(0)$.

We observe that  we can integrate by parts the last term in~\eqref{eqvort}
 another time using
the boundary condition for $u$ only: $u\cdot n\vert_{\partial\Omega}=0$.
Namely, for every $\varphi\in C^\infty(\bar\Omega)$ it holds
$$
\int_\Omega u\nabla\omega\,\varphi dx=-\int_\Omega \omega \div(u\varphi) dx,
$$
where $\varphi$ does not necessarily vanish at $\partial\Omega$,
we use $u\cdot n\vert_{\partial\Omega}=0$ instead.

The above argument shows that if $u$ is a weak solution
of the Euler system~\eqref{E}, then $\omega=\curl u$ satisfies
the following integral identity:
\begin{equation}\label{identomega}
\aligned
-\int_0^T\int_\Omega\omega\varphi\eta'(t)dxdt-&\int_0^T\int_\Omega\omega\div(u\varphi)\eta(t)dxdt+\\
+r&\int_0^T\int_\Omega\omega\varphi\eta(t)dxdt=\int_0^T\int_\Omega G\varphi\eta(t)dxdt,
\endaligned
\end{equation}
holding for all $\varphi\in C^\infty(\bar\Omega)$.

We now extend $\omega$ by zero outside $\Omega$  setting for all $t$
$$
\widetilde\omega=\left\{
               \begin{array}{ll}
                 \omega, & \hbox{in $\Omega$;} \\
                 0, & \hbox{in $\Omega^c=\mathbb{R}^2\setminus\Omega$.}
               \end{array}
             \right.
$$
In the similar way by define $\widetilde G$. The vector
function $u$ is extended
to a $\widetilde u\in \mathbf{H}^1(\mathbb{R}^2)$ in a
certain way that will be specified later. Since $\varphi$
in~\eqref{identomega} is an arbitrary smooth function
in~$C^\infty(\bar\Omega)$, it follows that
the following integral identity holds in the whole $\mathbb{R}^2$
\begin{equation}\label{identomegaR2}
\aligned
-\int_0^T\int_{\mathbb{R}^2}\widetilde\omega\varphi\eta'(t)dxdt-
&\int_0^T\int_{\mathbb{R}^2}\widetilde\omega\div(\widetilde u\varphi)\eta(t)dxdt+\\
+r&\int_0^T\int_{\mathbb{R}^2}\widetilde\omega\varphi\eta(t)dxdt=
\int_0^T\int_{\mathbb{R}^2} \widetilde G\varphi\eta(t)dxdt,
\endaligned
\end{equation}
holding for all $\varphi\in C^\infty_0(\mathbb{R}^2)$, $\eta\in C^\infty_0(0,T)$.

In other words, we have shown that $\widetilde \omega$ is a weak solution
in the whole $\mathbb{R}^2$ of  the equation
\begin{equation}\label{omegainR2}
\aligned
\partial_t\widetilde\omega+\widetilde u\nabla\widetilde \omega+r\widetilde\omega=\widetilde G,\\
\widetilde\omega(0)=\widetilde\omega^0.
\endaligned
\end{equation}

We shall now specify the construction of $\widetilde u$. Recall that
$$
u=u_0\oplus u_c, \quad u_0\in H_1, \ u_c\in H_c\,,
$$
where $u\in L^\infty(0,T;\mathcal{H}^1)$, and where
 $u_0$ has a single valued stream function~$\psi_0$:
$ u_0=\nabla^\perp\psi_0$, $\psi_0\in H^2(\Omega)$
(we do not use the additional information that $\psi_0=0$ at $\Gamma$).
In view of Lemma~\ref{L:basis}, so does $u_c$:
$u_c=\nabla^\perp\psi_c$, where $\psi_c\in H^2(\Omega)$ (at least).
We set $\psi=\psi_0+\psi_c$ and apply the extension operator $E$: $\widetilde \psi=E\psi$,
$$
\widetilde\psi\in H^2(\mathbb{R}^2),\quad
\|\widetilde\psi\|_{H^2(\mathbb{R}^2)}\le c(\Omega)
\|\psi\|_{H^2(\Omega)}.
$$
Then $\widetilde u:=\nabla^\perp\widetilde\psi$ is the required
extension of the vector function $u$ with
$$
\|\widetilde u\|_{\mathbf{H}^1(\mathbb{R}^2)}\le
c(\Omega)\|u\|_{{H}^1}, \qquad \div \widetilde u=0 \ \  \text{in the whole $\ \mathbb{R}^2$}.
$$

We are now in a position to apply the theory developed
in~\cite{DiP-Lio}. In particular, it follows
from~\cite[Theorem~II.3]{DiP-Lio} that the weak solution
$\widetilde \omega$ of~\eqref{omegainR2} in the sense~\eqref{identomegaR2}
is a renormalized solution, that is, satisfies
$$
\partial_t\beta(\widetilde\omega)+
\widetilde u \nabla\beta(\widetilde\omega)+
r\widetilde\omega\beta'(\widetilde\omega)=\beta'(\widetilde\omega)\widetilde G
$$
for all $\beta \in C^1_b(\mathbb{R})$ with $\beta(0)=0$.
This gives that
$$
\frac d{dt}\int_{\mathbb{R}^2}\beta(\widetilde\omega)dx+
r\int_{\mathbb{R}^2}\widetilde\omega\beta'(\widetilde\omega)dx
=\int_{\mathbb{R}^2}\widetilde G\beta'(\widetilde\omega)dx.
$$
Since $\beta(0)=0$ and $\widetilde \omega=0$ outside $\Omega$,
 the last equation goes over to
$$
\frac d{dt}\int_{\Omega}\beta(\omega)dx+
r\int_{\Omega}\omega\beta'(\omega)dx
=\int_{\Omega}G\beta'(\omega)dx.
$$
Choosing now for $\beta$ appropriate approximations of the function $s\to s^2$
we finally obtain
\begin{equation}\label{DPL31}
\frac12\frac d{dt}\|\omega(t)\|^2+
r\|\omega(t)\|^2=(\omega(t),G).
\end{equation}
Thus, we have proved the following result.
\begin{theorem}\label{T:en-ens}
Every weak solution of the damped and driven Euler equation
is of class $C([0,T];\mathcal{H}^1)$ and satisfies the equation of balance
of energy and enstrophy
\begin{equation}\label{enensbal}
\frac12\frac d{dt}\|u\|^2_1+
r\|u\|^2_1=(u,g)_1.
\end{equation}
\end{theorem}
\begin{proof}
The equation of balance~\eqref{enensbal} follows from~\eqref{balEnergy}
and~\eqref{DPL31}. The continuity in $\mathcal{H}^1$ follows from the continuity
in $\mathcal{H}$ (and, hence, weak continuity in $\mathcal{H}^1$) and the continuity
of the norm $t\to\|\omega(t)\|^2$, which follows from~\eqref{DPL31},
see~\cite{TemNS}.
\end{proof}

\section{Global attractor for the damped Euler system}
\label{Sec3}

For every solution of the damped Euler system we obtain from~\eqref{enensbal}
that
$$
\frac d{dt}\|u\|_1^2+2r\|u\|_1^2=2(g,u)\le2\|g\|_1\|u\|_1\le r\|u\|_1^2+r^{-1}\|g\|_1^2,
$$
so that by the Grownwall inequality
$$
\|u\|_1^2\le\|u(0)\|_1^2+r^{-2}\|g\|_1^2(1-e^{-rt})
$$
the ball~\eqref{ball} is also the absorbing ball
for the generalized semigroup of solution operators
$$
S(t)u^0=\{u(t)\}
$$
for the damped Euler system, where $\{u(t)\}$ is the section at time $t$
of all weak solutions with $u(0)=u^0$.

Our goal is to show that the generalized semigroup $S(t)$ has a weak
$(\mathcal{H}^1,\mathcal{H}^1_w)$ attractor in the sense of the
following definition (see~\cite{BV_MS}, \cite{BV}).
\begin{definition}\label{D:weakattr} A set $\mathscr A\subset{\mathcal H}^1$ is called an
$(\mathcal{H}^1,\mathcal{H}^1_w)$ attractor of the generalized semigroup $S(t)$ if
\par
1)  $\mathscr A$ is compact in the weak topology $\mathcal{H}^1_w$.
\par
2) $\mathscr A$ is strictly invariant: $S(t)\mathscr A=\mathscr A$.
\par
3) $\mathscr A$ attracts in the weak topology $\mathcal{H}^1_w$ bounded sets in $\mathcal{H}^1$.
\end{definition}

We first show that $S(t)$ a semigroup in the generalized sense.

\begin{lemma}\label{L:semi}
The family $S(t)$ has the semigroup property
\begin{equation}\label{semi}
S(t+\tau)u^0=S(t)S(\tau)u^0
\end{equation}
in the sense of the equality of sets.
\end{lemma}
\begin{proof}
The inclusion $S(t+\tau)u^0\subset S(t)S(\tau)u^0$
holds since every solution in the sense of Definition~\ref{D:weakE}
on the interval $[0,T]$ is also a solution on every  smaller interval $[\tau,T]$.
Let us prove the converse inclusion:
\begin{equation}\label{convincl}
S(t)S(\tau)u^0\subset S(t+\tau)u^0
\end{equation}

Any solution $u(t)$ satisfies on $[0,T]$ the integral identity
\begin{equation}\label{identEE}
\aligned
-\int_0^T(u,v\eta'(t))dt+\int_0^Tb(u,u,v\eta(t))dt+
r\int_0^T(u,v\eta(t))dt-\\-\int_0^T(g,v\eta(t))dt=
(u(0),v\eta(0))-(u(T),v\eta(T))
\endaligned
 \end{equation}
for every $v\in\mathcal{H}^1$ and $\eta\in C^\infty[0,T]$.
If this identity holds on the intervals $[0,\tau]$ and
$[\tau,t+\tau$, then adding them we see that it
holds on $[0,t+\tau]$ for every $\eta\in C^\infty[0,t+\tau]$.
This proves~\eqref{convincl}.
\end{proof}

The generalized semigroup is not known to be continuous
(the uniqueness is not proved), however,
the following two properties of it are, in a sense, a substitution
for the continuity and make it possible to
construct a weak attractor~\cite{BV_MS}, \cite{BV}.

\begin{lemma}\label{L:cont}
The generalized semigroup  $S(t)$ satisfies the following:
\begin{itemize}
  \item [1)]  $[S(t)X]_w\subset S(t)[X]_w\ \text{for any}\ X\subset B_0$,
  \item [2)] for every $y\in \mathcal{H}^1$ the set $S(t)^{-1}y\cap B_0$ is compact in $\mathcal{H}^1_w$.
\end{itemize}
Here $B_0$ is the absorbing ball~\eqref{ball}, and $[\ ]_w$ is the closure
in~$\mathcal{H}^1_w$.
\end{lemma}
\begin{proof}
1) Let $u=u^T\in [S(T)X]_w$.  Then there exists a sequence
$x_n\in X$ such that $S(T)x_n\to u^T$ weakly in $\mathcal{H}^1_w$.
The sequence $\{x_n\}$ is bounded in $\mathcal{H}^1$ and contains
a subsequence weakly converging to $x_0\in[X]_w$.

The set of all solutions $u_n(t)=S(t)x_n$ is bounded in
$C([0,T];\mathcal{H}^1)$, where the set $\partial_tu_n$ is
bounded in $L^\infty(0,T;L^{2-\varepsilon}(\Omega))$. Therefore we
can extract a subsequence $u_n$ such that
\begin{equation}\label{conv1}
u_n\to u\ \text{$*$-weakly in}\ L^\infty(0,T;\mathcal{H}^1)\ \text{and strongly in}\
L^2(0,T;\mathcal{H}).
\end{equation}

Each $u_n$ satisfies~\eqref{identEE}:
\begin{equation*}
\aligned
-\int_0^T(u_n,v\eta'(t))dt+\int_0^Tb(u_n,u_n,v\eta(t))dt+
r\int_0^T(u_n,v\eta(t))dt-\\-\int_0^T(g,v\eta(t))dt=
(x_n,v\eta(0))-(S(T)x_n,v\eta(T)).
\endaligned
 \end{equation*}
The convergence~\eqref{conv1} makes it possible to
pass to the limit in the integral terms, while by hypotheses we have
$$
(S(T)x_n,v\eta(T))\to(u^T,v\varphi(T)),\qquad(x_n,v\eta(T))\to(u^T,v\eta(T)).
$$
This proves 1), since  $u$ is a solution with $u(0)=x_0$ and $u(T)=u^T$, where
$x_0\in[X]_w$.

\noindent
2) The second property is proved similarly. Let
$$
u_n(0)=x_n,\ u_n(t)=y,\ x_n\in B_0,\ x_n\to x\in B_0\ \text{weakly in }\ \mathcal{H}^1.
$$
Passing to the limit  as in in part 1) we obtain that
the limiting function $u$ is a solution with
$u(0)=x$, $u(t)=y$, $x\in B_0$.
\end{proof}

This lemma shows that  the hypotheses
of \cite[Theorem~6.1]{BV_MS} or \cite[Theorem~II.1.1]{BV}
are satisfied for the generalized semigroup~$S(t)$. As a result we
have proved the existence of the weak attractor.
\begin{theorem}\label{T:Eweakattr}
The generalized semigroup $S(t)$ corresponding to the damped
Euler system has a weak $(\mathcal{H}^1,\mathcal{H}^1_w)$-attractor
$\mathscr A$.
\end{theorem}

Our next goal is to show that the attractor $\mathscr A$ is in fact
a (strong) global attractor in the sense of Definition~\ref{D:attr},
the only difference being that the semigroup $S(t)$ now
is a generalized (multi-valued) semigroup. The key role below is played by the
equation of balance of energy and enstrophy~\eqref{enensbal}.

\begin{theorem}\label{T:Estrongattr}
The attractor $\mathscr A$ is the (strong) global attractor.
\end{theorem}
\begin{proof}
We have to prove the asymptotic compactness  of $S(t)$, that is,
for every sequence $\{u_n^0\}$ bounded in $\mathcal{H}^1$ and
every sequence $t_n\to+\infty$ the sequence (of sets) $S(t_n)u_n^0$ is
precompact in~$\mathcal{H}^1$.

Let $u_n(t)$, $t\ge-t_n$ be a sequence of solutions  of the
damped Euler system:
$$
\left\{
\begin{array}{l}
\partial _{t}u_n+(u_n,\nabla )u_n+\nabla p_n+ru_n=g(x), \\
\div u_n=0,\quad u_n\vert_{t=-t_n}=u^0_n.
\end{array}
\right.
$$
Then $u_n(0)\in S(t_n)u_n^0$ and we have to verify that
$\{u_n(0)\}_{n=0}^\infty$ is precompact in~$\mathcal{H}^1$.

The solutions $u_n(t)$, $t\ge-t_n$, are bounded in $C_b([-T,\infty),\mathcal{H}^1)$
for $T\le t_n$ and we can extract a subsequence
$$
u_n(0)\to \bar u\in \mathcal{H}^1\ \  \text{weakly in} \ \  \mathcal{H}^1.
$$
Along a further subsequence we have
$$
u_n\to u\ \text{$*$-weakly in}\ L^\infty(-T,T;\mathcal{H}^1)\ \text{and strongly in}\
L^2(-T,T;\mathcal{H}).
$$
This is enough to pass to the limit in the integral identities satisfied
by $u_n$ to obtain that the following integral identity holds for $u$:
$$
\aligned
-\int_{\mathbb{R}}(u,v\eta'(t))dt+&\int_{\mathbb{R}}b(u,u,v\eta(t))dt+\\
r&\int_{\mathbb{R}}(u,v\eta(t))dt-\int_{\mathbb{R}}(g,v\eta(t))dt=0,
\quad \eta\in C^\infty_0({\mathbb{R}}),
\endaligned
$$
which gives that $u$ is a solution of the damped Euler system
bounded on $t\in\mathbb{R}$. Next, we have
\begin{equation}\label{u=u(0)}
u(0)=\bar u.
\end{equation}
This is standard \cite{TemNS}. On one hand, for $\eta(0)\ne0$ we have
\begin{equation}\label{inttime}
\aligned
-\int_{-\infty}^0(u,v\eta'(t))dt+&\int_{-\infty}^0b(u,u,v\eta(t))dt+\\
r&\int_{-\infty}^0(u,v\eta(t))dt-\int_{-\infty}^0(g,v\eta(t))dt=-(\bar u,v)\eta(0),
\endaligned
\end{equation}
On the other hand, multiplying the equation
$$
\frac d{dt}(u,v)+b(u,u,v)+r(u,v)=(g,v)
$$
by the same $\eta$ and integrating from $-\infty$ to $0$
we obtain equality~\eqref{inttime} with the right-hand side equal to
$-( u(0),v)\eta(0)$. This gives~\eqref{u=u(0)}.

Thus, we have that $u_n(0)\to u(0)$ weakly in $\mathcal{H}^1$,
we now show that $u_n(0)\to u(0)$ strongly in $\mathcal{H}^1$.
We multiply the balance equation~\eqref{enensbal} for $u_n$ by
$e^{2rt}$ and integrate from $-t_n$ to $0$. We obtain
$$
\|u_n(0)\|_1^2=\|u_n(-t_n)\|_1^2e^{-2rt_n}+2\int_{-t_n}^0(u_n(t),g)_1e^{2rt}dt.
$$
Since $u_n(-t_n)$ are uniformly bounded in~$\mathcal{H}^1$
and
$$ u_n\to u\quad \text{$*$-weakly in}\quad L^\infty_{loc}(\mathbb{R};\mathcal{H}^1)$$
we can pass to the limit as $n\to\infty$ to obtain
$$
\lim_{n\to\infty}\|u_n(0)\|_1^2=2\int_{-\infty}^0(u(t),g)_1e^{2rt}dt.
$$
The complete trajectory $u(t)$ also satisfies the balance equation,
and acting similarly we obtain
$$
\|u(0)\|_1^2=2\int_{-\infty}^0(u(t),g)_1e^{2rt}dt.
$$
Thus, we have shown that
$$
\lim_{n\to\infty}\|u_n(0)\|_1^2=\|u(0)\|_1^2,
$$
which along with the established weak convergence
gives that
$$
u_n(0)\to u(0)\quad\text{strongly in} \quad \mathcal{H}^1,
$$
and completes the proof.
\end{proof}

\section{Upper semi-continuity of the attractors
in the limit of vanishing viscosity}
\label{Sec4}

In this concluding section we study the dependence
of the attractors $\mathscr A_\nu$ of the damped
Navier--Stokes system on the viscosity coefficient
$\nu$ as $\nu\to0^+$. In the previous section we have
shown that the damped Euler system (with $\nu=0$) has
the global attractor
$$
\mathscr A_{\nu=0}=:\mathscr A_0.
$$
Furthermore, uniformly for $\nu\ge0$ the following estimate holds:
$$
\sup_{u\in\mathscr A_\nu}\|u\|_1\le\frac{\|g\|_1}r.
$$

\begin{theorem}\label{T:strong_upper_cont}
The attractors $\mathscr{A}_\nu$
depend upper semi-continuously on  $\nu$ as $\nu\to0^+$.
In other words
\begin{equation}\label{strong-cont}
\lim_{\nu\to0^+}\dist_{\mathcal H^1}(\mathscr{A}_\nu,\mathscr{A}_0)=0,
\end{equation}
where
\begin{equation}\label{defdist}
\dist_{\mathcal H^1}(X,Y):=
\sup_{x\in X}\Inf_{y\in Y}\|x-y\|_{\mathcal H^1}.
\end{equation}
\end{theorem}
\begin{proof}
We  take  an arbitrary sequence $\nu_n\to0^+$, and for every
$\nu_n$ choose a point on the attractor
$\mathscr{A}_{\nu_n}$ of equation \eqref{NS} with $\nu=\nu_n$. Specifically,
we choose the point on $\mathscr{A}_{\nu_n}$, whose distance from $\mathscr{A}_{0}$ is equal
to the distance from $\mathscr{A}_{\nu_n}$ to $\mathscr{A}_{0}$. In view of the compactness
of   $\mathscr{A}_{\nu_n}$ and  $\mathscr{A}_{0}$ such a point exists.
These points lie on~$\mathscr{A}_{\nu_n}$ and therefore there are complete
trajectories passing through them, and we can denote these points by $u_n(0)$,
so that
\begin{equation}\label{unibd}
u_n(0)\in\mathscr{A}_{\nu_n}, \quad u_n\in C_b(\mathbb{R},\mathcal H^1),
\quad \|u_n\|_{C_b(\mathbb{R},\mathcal H^1)}\le r^{-1}\|g\|_1,
\end{equation}
and in view of our choice
\begin{equation}\label{dist}
\dist_{\mathcal{H}^1}(u_n(0),\mathscr{A}_{0})=
\dist_{\mathcal{H}^1}(\mathscr{A}_{\nu_n},\mathscr{A}_0).
\end{equation}
In view of~\eqref{unibd} we can extract a subsequence $u_{\nu_n}$ for which
for a $\bar u\in \mathcal{H}^1$
$$
u_{\nu_n}(0)\to \bar u\quad  \text{weakly in} \quad  \mathcal{H}^1\
\text{as}\quad n\to\infty,
$$
and along a further subsequence we have
$$
u_{\nu_n}(0)\to u_0\ \text{$*$-weakly in}\ L^\infty_{\mathrm{loc}}(\mathbb{R},\mathcal{H}^1)\
\text{and strongly in}\ L^2_{\mathrm{loc}}(\mathbb{R},\mathcal{H}).
$$

The solutions $u_{\nu_n}$, by definition, satisfy the integral identity
$$
\aligned
-\int_{\mathbb{R}}(u_{\nu_n},v\eta'(t))dt+\int_{\mathbb{R}}b(u_{\nu_n},u_{\nu_n},v\eta(t))dt+\\+
\int_{\mathbb{R}}\nu_n(\curl u_{\nu_n},\curl v\eta(t))dt+
r\int_{\mathbb{R}}(u_{\nu_n},v\eta(t))dt-\int_{\mathbb{R}}(g,v\eta(t))dt=0.
\endaligned
$$
We now pass to the limit in this  identity taking into account that
$$
\nu_n(\curl u_{\nu_n},\curl v)\to 0\ \text{as}\ \nu_n\to0,
$$
and obtain that $u_0$ is a solution (a complete trajectory) of the damped Euler
system and therefore satisfies the balance equation~\eqref{enensbal}.
In addition, as in Theorem~\ref{T:Estrongattr},  we can show that
$u(0)=\bar u$, so that
$$
u_{\nu_n}(0)\to u(0)\quad  \text{weakly in} \quad  \mathcal{H}^1.
$$

The complete trajectories $u_n=u_n(t)$ of the damped Navier--Stokes system~\eqref{NS} satisfy the balance equation~\eqref{NSee}. We drop there the second (non-negative) term
multiply the resulting inequality by $e^{2rt}$ and integrate from
$-t_n$ to $0$, where $t_n\to+\infty$. We obtain
\begin{equation*}
\|u_{\nu_n}(0)\| _{1}^{2}\le
\| u_{\nu_n}(-t_{n})\|_{1}^{2}e^{-2rt_{n}}+2\int_{-t_{n}}^{0}(u_{\nu_n}(t),g)_{1}e^{2rt}dt.
\end{equation*}
In the limit as $n\to\infty$ this gives that
$$
\limsup_{n\to\infty}\|u_{\nu_n}(0)\| _{1}^{2}\le
2\int_{-\infty }^{0}(u_0(t),g)_{1}e^{2rt}dt.
$$
For the solution $u_0$ as in Theorem~\ref{T:Estrongattr} we have
$$
\|u_{0}(0)\| _{1}^{2}=
2\int_{-\infty}^{0}(u_{0}(t),g)_{1}e^{2rt}dt,
$$
and together  with the previous inequality this gives that
\begin{equation}\label{limsup}
\limsup_{n\to\infty}\|u_{\nu_n}(0)\| _{1}^{2}\le
\|u_0(0)\|^2_1.
\end{equation}
Since by the weak convergence we always have
$$
\|u_0(0)\|_1\le\liminf_{n\to\infty}\|u_{n}(0)\| _{1},
$$
it follows from~\eqref{limsup} that
$$
\lim_{n\to\infty}\|u_{\nu_n}(0)\|_1=\|u_{0}(0)\| _{1},
$$
and, finally, that
$$
\lim_{n\to\infty}\|u_{\nu_n}(0)-u_0(0)\| _{1}=0.
$$

Taking into account~\eqref{dist} we obtain that
\begin{equation}\label{limdist}
\lim_{n\to\infty}\dist_{\mathcal H^1}(\mathscr{A}_{\nu_n},\mathscr{A}_0)=0.
\end{equation}
Since in the course of the proof we have been several times  passing to subsequences
we have actually shown that
\begin{equation}\label{liminfdist}
\liminf_{\nu_n\to0^+}\dist_{\mathcal H^1}(\mathscr{A}_{\nu_n},\mathscr{A}_0)=0
\end{equation}
for  \emph{any}  sequence  $\nu_n\to0^+$.
This obviously implies~\eqref{strong-cont}. The proof is complete.
\end{proof}

\begin{remark}
A similar result in $\mathbb{R}^2$ was recently obtained in~\cite{IChMZ}.
\end{remark}

\end{document}